\documentclass[12pt]{article}
\usepackage{amsmath}
\usepackage{amsfonts}
\usepackage{mathtools}
\usepackage{amssymb}
\usepackage{amsthm}
\usepackage[all,cmtip]{xy}
\newtheorem{theo}{Theorem}
\newtheorem{lemm}{Lemma}
\newtheorem{defi}{Definition}

\newtheorem{coro}{Corollary}

\newtheorem{exam}{Example}
\AtEndDocument{\bigskip{\footnotesize%
  \textsc{Department of Mathematics, University of California, Riverside, Riverside, California 92521 U.S.A.} \par  
  \textit{E-mail address}: \texttt{hyunseung@math.ucr.edu} \par
  \addvspace{\medskipamount}
  \textsc{Department of Mathematics, University of California, Riverside, Riverside, California 92521 U.S.A.} \par  
  \textit{E-mail address}: \texttt{awalk010@ucr.edu} \par
}}

\begin{document}

\baselineskip 20pt

\title{The radical of an n-absorbing ideal}
\author{Hyun Seung Choi and Andrew Walker}
\date{}
\maketitle

\begin{abstract}
In this note we show that in a commutative ring $R$ with unity, for any $n > 0$, if $I$ is an $n$-absorbing ideal of $R$, then $(\sqrt{I})^{n} \subseteq I$.
\end{abstract}

\begin{defi}
{\em An ideal $I$ of a commutative ring $R$ is
said to be {\bf n-absorbing} if whenever
$a_{1} \cdots a_{n+1} \in I$
for
$a_{1},\ldots,a_{n+1}
\in R$, then
$a_{1}\cdots a_{i-1}a_{i+1}\cdots a_{n+1} \in I$
for some $i \in \{1, 2, \ldots, n+1\}$.
%An ideal $I$ is said to be {\bf strongly n-absorbing}
%if whenever $I_1 I_2 \cdots I_{n+1} \subseteq I$
%for ideals $I_1, I_2, \ldots, I_{n+1}$ of $R$, then
%there are $n$ of the $I_i$'s whose product is in $I$.
}
\end{defi}

In \cite[Theorem $2.1(e)$]{Anderson},
it is shown that if
$a \in \sqrt{I} \text{ and } I
\text{ is } n\text{-absorbing, then } a^{n} \in I$.
Conjecture 2
in \cite[page 1669]{Anderson}
states that
more generally,
if $I$ is $n$-absorbing, then
$(\sqrt{I}\:)^{n} \subseteq I$.
That is, if $a_1, a_2, \ldots, a_n \in \sqrt{I}$,
then $a_1a_2 \cdots a_n \in I$.
The object of
this note is to prove this conjecture.
%This is the special case of strongly $n$-absorbing
%where $I_i
% = I_2 = \cdots = I_{n+1}
%= \sqrt{I}$ for each $i$.

Throughout, all rings will be assumed to be commutative and unital. If
$n$ is a positive integer, we'll consider $\mathbb{N}^{n}_{0}$ as a
partially ordered set with the lexicographic ordering. That is, if
$\alpha,\beta \in \mathbb{N}^{n}_{0}$, then $\alpha \geq_{\text{lex}}
\beta$ if the leftmost non-zero coordinate of $\alpha - \beta$ is
non-negative.

Our first
observation is that when considering the problem
of when $(\sqrt{I}\:)^n \subseteq I$ for $I$ $n$-absorbing,
we may assume without
loss of generality that $I = 0$.

\begin{lemm}
Suppose $(\sqrt{0})^{n} = 0$ in any ring such that $0$ is
$n$-absorbing. Then for an $n$-absorbing ideal $I$ in an arbitrary ring
$R$, $(\sqrt{I})^{n} \subseteq I$.

\end{lemm}

\begin{proof}
Let $R' = R/I$. Then $0$ is $n$-absorbing in $R'$ \cite[Theorem $4.2
(a)$]{Anderson}, so that $(\sqrt{0})^{n} = 0$. Let $f \colon R \to R'$
be the canonical map. Then $(\sqrt{I})^{n} = (f^{-1}(\sqrt{0}))^{n}
\subseteq f^{-1}((\sqrt{0})^{n}) = f^{-1}(0) = I$.
\end{proof}

\begin{lemm}
If $I$ is an $n$-absorbing ideal in a ring $R$ and $k \geq n$, then $I$
is a $k$-absorbing ideal of $R$.

\end{lemm}

\begin{proof}
\cite[Theorem $2.1(b)$]{Anderson}.
\end{proof}
Next we develop a technical result
involving linear maps.
If $m$
is any positive integer and $R$ any ring, then $e_{j} \in R^{m}$ refers
to the $j$-th canonical basis element $e_{j} =
[0,\ldots,1,\ldots,0]^t$
of $R^m$ (where the $t$ denotes transpose).
We denote by
$\pi_{j} : R^{m} \to R$
the canonical projection for each
$j = 1,\ldots,m$.

\begin{defi}
\normalfont
Let $R$ be a ring, $m \in \mathbb{N}$ and $\varphi \colon R^{m} \to
R^{m}$ an $R$-linear map. We'll say that $\varphi$ is
\textbf{projectively zero} if for any $v \in R^{m}$, $\pi_{i}\varphi(v)
= 0$ for some $i=1,\ldots,m$.

\end{defi}

In the following example,
we establish
a relationship between projectively zero maps and $n$-absorbing ideals.
Let's consider the simplest interesting case, when $0$ is a
$2$-absorbing ideal. We wish to show that
$(\sqrt{0}\,)^{2} = 0$.
That is, if $a,b \in \sqrt{0}$, then
$ab = 0$.
Consider
the
matrix \[ \left[ \begin{array}{cc}
ab & b^2 \\
a^2 & ab \\
\end{array} \right]. \]
We claim that this matrix represents a projectively zero map $\varphi
\colon R^{2} \to R^{2}$.
By Theorem
\cite[Theorem 2.1e]{Anderson},
we
know that $a^{2} =
b^{2} = 0$.
So that
the above matrix simplifies to
\begin{equation}
\label{equation1}
 \left[ \begin{array}{cc}
ab & 0 \\
0 & ab \\
\end{array} \right].
\end{equation}
Say
$v = ce_{1} + c'e_{2} \in R^{2}$, where
$c,c' \in R$. Then
$\varphi(v) = (cab)e_{1} + (c'ab)e_{2}$.
That is
$$
\left[
\begin{array}{cc}
ab & 0 \\
0 & ab \\
\end{array} \right]
\left[
\begin{array}{c}
c  \\
c'  \\
\end{array} \right]
=
\left[
\begin{array}{c}
cab  \\
c'ab  \\
\end{array} \right]
.$$
To show $\varphi$ is
projectively zero, we need
one of the monomials
$cab$ or $c'ab$ to be $0$.
We have
\begin{equation}
\label{equation2}
0 = a^{2}bc + b^{2}ac' = ab(ca + c'b)
\mbox{ since } a^2 = b^2 = 0.
\end{equation}

Since $0$ is $2$-absorbing and
$ab(ca + c'b)$ = $0$,
then at least one of
$ab$, $b(ca + c'b)$, or $a(ca + c'b)$ is zero.
If $ab = 0$,
then both
$\pi_1\varphi(v)$ =
$cab$
and
$\pi_2\varphi(v)$ =
$c'ab$ are zero.

If $b(ca + c'b) = 0$, then since
$b^{2} = 0$,
we get $0$ = $cab$
= $\pi_1\varphi(v)$.
Similarly, if
$a(ca + c'b)$ = $0$,
we get
$0$ = $c'ab$
= $\pi_2\varphi(v)$.
Thus
$\varphi$ is projectively zero.

This will be useful since Lemma $3$ below will tell us
that $ab  =
0$, and thus
$(\sqrt{0}\,)^{2} = 0$.

\begin{defi}
{\em We say that
a linear map
$\varphi \colon R^{m} \to R^{m}$ is \textbf{upper-triangular} if for
each $j = 1,\ldots, m$,
$\pi_{i}\varphi(e_{j}) = 0$ whenever $i > j$.
}

\end{defi}

Lemma
\ref{upper.triangular.zero.on.diagonal}
shows that certain upper-triangular matrices must have
at least one zero on their diagonal.

\begin{lemm}
\label{upper.triangular.zero.on.diagonal}
Suppose that $\varphi \colon R^{m} \to R^{m}$ is a projectively zero
upper-triangular map. Then $\pi_{j}\varphi(e_{j}) = 0$ for some $j$.

\end{lemm}

\begin{proof}
Let
$ j_{1}$
= $\max \{ i \in \{1, 2, \ldots, m\} \mid \pi_{i}\varphi(e_{m}) = 0 \}$.
Since
$\varphi$ is projectively zero, the above set is non-empty and so
$j_{1}$ is a positive integer.
Similarly we can define a
positive integer
$j_{2} = \max \{ i \in \{1, \ldots, m\} \mid \pi_{i}\varphi(e_{j_{1}} +
e_{m}) = 0 \}$.
Proceeding in the same way,
we have for each $k \in \mathbb{N}$,
a positive integer
$j_{k}$ with
\begin{equation}
\label{equation7}
j_{k} = \max \{ i \mid \pi_{i}\varphi(e_{j_{k-1}} +
\cdots + e_{j_{2}} + e_{j_{1}} + e_{m}) = 0\}.
\end{equation}

Suppose that for each $j \in \{1, \ldots, m\}$
that $\pi_{j}\varphi(e_{j}) \neq 0$.
We then claim that the sequence
$j_1, j_2, \ldots$
of positive integers
constructed above
is
strictly decreasing.
If not, then
for some $k \in \mathbb{N}$ we have either
$j_{k} < j_{k+1}$
or $j_{k}$ = $j_{k+1}$.
Suppose that $j_{k} < j_{k+1}$.
Now by definition
of $j_{k+1}$, we have
$$
0 = \pi_{j_{k+1}}\varphi(e_{j_{k}} + e_{j_{k-1}} +
\cdots + e_{j_{1}} + e_{m}) =$$
\begin{equation}
\label{equation8}
 \pi_{j_{k+1}}\varphi(e_{j_{k}}) +
\pi_{j_{k+1}}\varphi(e_{j_{k-1}}) + \cdots +
\pi_{j_{k+1}}\varphi(e_{j_{1}}) + \pi_{j_{k+1}}\varphi(e_{m})
\end{equation}
and the first term in (\ref{equation8}) is zero since $j_k < j_{k+1}$
and $\varphi$ is upper triangular.
So this is
$$
= 0 + \pi_{j_{k+1}}\varphi(e_{j_{k-1}}) + \cdots +
\pi_{j_{k+1}}\varphi(e_{j_{1}}) + \pi_{j_{k+1}}\varphi(e_{m}) =
\pi_{j_{k+1}}\varphi(e_{j_{k-1}} + \cdots + e_{j_{1}} + e_{m}).$$
But this contradicts how $j_{k}$ was defined
in equation
(\ref{equation7}).
 So the only way for
$j_{k+1} \geq j_{k}$ to happen is if $j_{k+1} = j_{k}$. But then
$$ 0 =
\pi_{j_{k+1}}\varphi(e_{j_{k}} + e_{j_{k-1}} + \cdots + e_{j_{1}}
+ e_{m})
 = \pi_{j_{k}}\varphi(e_{j_{k}} + e_{j_{k-1}} + \cdots +
e_{j_{1}} + e_{m})$$
$$ = \pi_{j_{k}}\varphi(e_{j_{k}}) +
\pi_{j_{k}}\varphi(e_{j_{k-1}} + \cdots + e_{j_{1}} + e_{m})
 =
\pi_{j_{k}}\varphi(e_{j_{k}}) + 0 = \pi_{j_{k}}\varphi(e_{j_{k}}),$$
which contradicts our
assumption
that $\pi_{j}\varphi(e_{j}) \neq 0$ for any $j$.
Thus the $\{j_{k}\}$ form
a strictly decreasing sequence,
a contradiction
since $j_k \in \{1, \ldots, m\}$ for each $k$.
\end{proof}

We will need some partial orderings on monomials.

\begin{defi}
\normalfont
Let $x_{1},\ldots,x_{n}$ be indeterminates over a ring $R$. The
\textbf{(unordered) multi-degree} of a monomial $M =
x^{k_{1}}_{1}\cdots x_{n}^{k_{n}}$ in $R[x_{1},\ldots,x_{n}]$
is the $n$-tuple
$\alpha = (k_{\sigma(1)},\ldots, k_{\sigma(n)}) \in
\mathbb{N}^{n}_{0},$
where $\sigma$ is a permutation of $\{1,\ldots,n\}$ such that
$k_{\sigma(1)} \geq
\cdots \geq k_{\sigma(n)}.$ Denote this $n$-tuple by
$\text{multideg}(M)$. We'll also write $|\alpha|$ for the
\textbf{degree} $\sum_{i=1}^n k_{i}$ of the monomial $M$.

\end{defi}

\begin{exam}
\normalfont
Suppose $x,y,z$ are indeterminates over $R$. Then
\[\text{multideg}(x^{2}y^{4}z^{2}) = \text{multideg}(x^{4}y^{2}z^{2}) =
 (4,2,2).\]
\end{exam}

Suppose $J  =
(a_{1},\ldots,a_{n})R$ is a finitely generated ideal of a ring $R$. If
$\underline{x} = x_{1},\ldots,x_{n}$ is a sequence of indeterminates
over $R$, we have a natural $R$-algebra homomorphism
$f : R[\underline{x}] \to R$,
where
$x_{i} \mapsto a_{i}$.
 Let $H$
= $(\underline{x})R[\underline{x}]$. Then under this
map, $f(H) = J$.
Moreover, for any $k \in \mathbb{N}$, we
have
$f(H^k) = J^{k}$.
Then $H^k$
is just the ideal of $R[\underline{x}]$ generated by
all monomials $M$ in $R[\underline{x}]$ of degree $k$.
Now grouping
together all monomials of degree $k$ that have the same (unordered)
multi-degree, we may write
$$H^{k} = \sum_{\alpha \in
\mathbb{N}^{n}_{0},|\alpha| = k}
H_\alpha^k,$$
where
$H_\alpha^k$
is the ideal of $R[\underline{x}]$
generated by all monomials $M$ with $\deg(M) = k$ and
$\text{multideg}(M) = \alpha$.
Thus $J^{k}$ =
$f(H^k)$ =
$f(\sum H_\alpha^k)$
= $\sum f(H_\alpha^k)$.
For $\alpha \in \mathbb{N}^{n}_{0}$ with
$|\alpha| = k$,
let $J^{k}_{\alpha}  =
f(H_\alpha^{k})$.
So that we may write \[ J^{k}  =
\sum_{\alpha \in \mathbb{N}_{0}^{n}, |\alpha| = k} J_{\alpha}^{k}. \]

\begin{exam}
Let $x,y,z$ be indeterminates over a ring $R$. Then in the above notation,
\[ H^{3} = H_{(3,0,0)}^{3} + H_{(2,1,0)}^{3} + H_{(2,0,1)}^{3} + H_{(1,2,0)}^{3} + H_{(1,1,1)}^{3} + H_{(1,0,2)}^{3}  + H_{(0,3,0)}^{3}   + H_{(0,2,1)}^{3}  + H_{(0,1,2)}^{3}  + H_{(0,0,3)}^{3}  \]\[ = (x^{3},y^{3},z^{3}) + (x^{2}y,x^{2}z,y^{2}x,y^{2}z,z^{2}x,z^{2}y) + (0) + (0) + (xyz)  + (0) + (0) + (0) + (0) + (0). \]
For instance, $H_{(2,0,1)}^{3} = 0$ since there are no monomials with (unordered) multi-degree $(2,0,1)$; the (unordered) multi-degree of a monomial $M \in R[x,y,z]$ is always of the form $(n,m,\ell)$, where $n \geq m \geq \ell$. 
\end{exam}

Using this notation, we are now ready to prove the main conjecture.

\begin{theo} \label{main_theorem}
Let $0$ be an $n$-absorbing ideal in a ring $R$.
Then $\left(\sqrt{0}\:\right)^{n} = 0$.

\end{theo}

\begin{proof}
We assume $n > 1$, since the $n = 1$ case is trivial. Fix $a_{1},\ldots, a_{n} \in \sqrt{0}$ and let $J  =
(a_{1},\ldots,a_{n})R$. Observe that $a_{1} \cdots a_{n} \in J_{(1,1,\ldots, 1)}^{n}$, so that it suffices to show $J_{(1,1,\ldots,1)}^{n} = 0$. Even better, we aim to show \begin{equation} \label{inductionsetup}
J_{\alpha}^{k} = 0 \text{ for all }\alpha \in \mathbb{N}^{n}_{0} \text{ with }|\alpha| = k \geq n.
\end{equation} Since $a^{n}_{i} = 0$ for all $i \in \{1,\ldots, n\}$, we have $J^{k}_{\alpha} = 0$ for all $\alpha \in \mathbb{N}^{n}_{0}$ with $|\alpha| = k \geq n^{2} - n +1$. To prove (\ref{inductionsetup}), it thus remains to show \begin{equation} \label{inductionsetup2}
J_{\alpha}^{k} = 0 \text{ for all } \alpha \in \Delta \text{ with }|\alpha| = k,\end{equation} where \[\Delta := \{ \alpha \in \mathbb{N}^{n}_{0} \mid n^{2} - n \geq |\alpha| \geq n  \}.\] Now for $\alpha, \beta \in \mathbb{N}^{n}_{0}$, write $\beta \succeq \alpha$ if one of the following holds: \begin{enumerate}
\item $|\beta| > |\alpha|$ or
\item $|\beta| = |\alpha|$ and $\beta \geq_{\text{lex}} \alpha$.
\end{enumerate}

It follows that $\succeq$ defines a total ordering on $\Delta$. We prove that (\ref{inductionsetup2}) holds by means of an induction on $\Delta$ with respect to the total ordering $\succeq$. The largest element of $\Delta$ (with respect to $\succeq$) is $\gamma$, where $\gamma = (n^{2}-n,0,0\ldots,0)$. So \[ J_{\gamma}^{n^{2}-n} = (a^{n^{2}-n}_{1},a^{n^{2}-n}_{2},\ldots, a^{n^{2}-n}_{n}) = 0,\] since $n^{2} - n \geq n$ and $a^{n}_{i} = 0$ for all $i \in \{1,\ldots,n\}$. Now, say $\alpha \in \Delta$ with $|\alpha| = k$ and assume that $J_{\beta}^{s} = 0$ for any $\beta \succ\alpha$ with $\beta \in \Delta$ and $|\beta| = s$. We prove $J_{\alpha}^{k} = 0$.

Recall that $J^{k}_{\alpha}$ is generated by elements of the form
$g  = f(M)$,
where $M$ is a monomial of $R[\underline{x}]$ with
$\text{multideg}(M) = \alpha$ and $|\alpha| = k$. So write $g =
a^{k_{1}}_{\ell_{1}} \cdots a_{\ell_{m}}^{k_{m}}$, where each
$k_{t} > 0$, $\Sigma_{t=1}^{m}k_{t}=k$, 
and each $a_{\ell_{j}}$ is a distinct element of
$\{a_{1},\ldots,a_{n}
\}$. Set $y_{j} = a_{\ell_{j}}$ for each $j \in \{1,\ldots, m\}$
and we may assume
without loss of generality,
that $k_t \geq k_{t+1}$ for each $t$.
So $g = y^{k_{1}}_{1} \cdots y_{m}^{k_{m}}$.

Let $\textbf{C}$ be the $m \times m$ matrix
$\Big( \displaystyle \frac{y_{j}}{y_{i}} g \Big)_{i,j}$. Note that $\frac{y_{j}}{y_{i}} g$ is an element of $R$, since $k_{i}$ is positive.
So
$$
C =
\left[
\begin{array}{ccccc}
\frac{y_1}{y_1}g    & \frac{y_2}{y_1}g   & \cdots &\frac{y_{m-1}}{y_1}g    &\frac{y_m}{y_1}g \\
\frac{y_1}{y_2}g    & \frac{y_2}{y_2}g   & \cdots &\frac{y_{m-1}}{y_2}g    &\frac{y_m}{y_2}g \\
\vdots              & \vdots             & \vdots & \vdots                 & \vdots          \\
\frac{y_1}{y_{m-1}}g&\frac{y_2}{y_{m-1}}g& \cdots &\frac{y_{m-1}}{y_{m-1}}g&\frac{y_{m-1}}{y_m}g \\
\frac{y_1}{y_m}g    &\frac{y_2}{y_m}g    & \cdots & \frac{y_{m-1}}{y_m}g   &\frac{y_m}{y_m}g
\end{array}
\right]
=
\left[
\begin{array}{ccccc}
\frac{y_1}{y_1}g & \frac{y_2}{y_1}g & \cdots &\frac{y_{m-1}}{y_1}g  & \frac{y_m}{y_1}g \\
 0               & \frac{y_2}{y_2}g & \cdots &\frac{y_{m-1}}{y_2}g &\frac{y_m}{y_1}g\\
\vdots           & \vdots           & \vdots & \vdots              &\vdots          \\
0                & 0                & \cdots & \frac{y_{m-1}}{y_{m-1}}g&\frac{y_m}{y_{m-1}}g \\
0                & 0                & \cdots & 0                   &\frac{y_m}{y_m}g
\end{array}
\right].
$$
Indeed if $i > j$
we may write
$\displaystyle \frac{y_{j}}{y_{i}} g = f(M')$,
where
\begin{displaymath}
M'
= \frac{x_{\ell_{j}}}{x_{\ell_{i}}}(x_{\ell_{1}}^{k_1}\cdots x_{\ell_{j}}^{k_j} \cdots x_{\ell_{i}}^{k_i}\cdots x_{\ell_{m}}^{k_m})
= x_{\ell_{1}}^{k_1}\cdots x_{\ell_{j}}^{k_j+1} \cdots x_{\ell_{i}}^{k_i-1}\cdots x_{\ell_{m}}^{k_m}
\end{displaymath}
is a
monomial of $R[\underline{x}]$ with $ \beta = \text{multideg}(M')
>_{\text{lex}} \text{multideg}(M) = \alpha$ and $|\beta| = |\alpha| = k$. Thus $\beta \in \Delta$ with $\beta \succ \alpha$, and hence $\displaystyle
\frac{y_{j}}{y_{i}} g \in J_{\beta}^{k} = 0$.
So
$\textbf{C}$
is upper-triangular.
Let $\varphi \colon R^{m} \to R^{m}$ be the
$R$-linear map defined by
$v \mapsto \textbf{C}v$.
Then $\varphi$ is upper triangular.
Moreover, $\varphi$ is projectively zero.
Indeed, given any
$v = \sum c_{j}e_{j} \in R^m$
we have that for each $i$,
\[ \pi_{i}\varphi(v) = \sum^{m}_{j=1} c_{j}\frac{y_{j}}{y_{i}}g. \]
On the other hand,
we note $J^{k+1} = 0$ by our induction hypothesis (or by our previous remarks if $k = n^{2} - n$), so $g \in (J^{k+1} \colon J) =
(0 \colon J)$.
Then $\displaystyle g \Big(\sum^{m}_{j=1}
c_{j}y_{j}\Big) \in gJ = 0$.
Now since $0$ is $n$-absorbing and $g$ is the product of $k
\geq n$ elements, we must have that for some $i$ (if $g$ is not zero),
\[ 0 =
\frac{1}{y_{i}}g \Big(\sum^{m}_{j=1} c_{j}y_{j}\Big) =
\sum^{m}_{j=1} c_{j}\frac{y_{j}}{y_{i}}g = \pi_{i}\varphi(v). \] So
$\varphi$ is a projectively zero upper-triangular map.
Thus by Lemma
\ref{upper.triangular.zero.on.diagonal},
$\pi_{j}\varphi(e_{j}) = 0$ for some $j$. But
$\pi_{j}\varphi(e_{j}) = \displaystyle\frac{y_{j}}{y_{j}}g = g$.
Thus $g$ = $0$ and the induction is complete.
\end{proof}

\begin{coro}
If $I$ is $3$-absorbing with $\sqrt{I} = P$ a prime ideal and
$x \in P$, then $I_{x} = (I :_{R} x)$ is a $2$-absorbing ideal of $R$.

\end{coro}

\begin{proof}
We must show that if $abc \in I_{x}$, then $ab,ac$ or $bc \in I_{x}$.
Since $I$ is $3$-absorbing and $abcx \in I$,
then
either $abc \in I$,
$abx \in I$, $acx \in I$
or $bcx \in I$. So we assume $abc \in I$.

Without loss of generality, we can assume $a \in P$ as well. Since $P^{3} \subseteq I$ by Theorem \ref{main_theorem},
$xbc(a + x^2) \in I$, so that since $I$ is $3$-absorbing, we're left with $4$ possibilities: 
$xbc, xc(a+x^{2}), xb(a+x^{2})$, or $bc(a+x^{2}) \in I$. From the first three choices,
we can conclude $xbc, xca,$ or $xba \in I$ respectively, so that we may assume $bc(a+x^{2}) \in I$,
from which it follows that $bcx^{2} \in I$. Again since $I$ is $3$-absorbing, this implies $bcx$, $bx^{2}$
or $cx^{2} \in I$. So we may deduce $bx^{2}$ or $cx^{2} \in I$. If $bx^{2} \in I$, then $abx(x+ c) \in I$ implies 
that one of $abx, ab(x+c), bx(x+c),$ or $ax(x+c) \in I$. In any of these cases, we can deduce either $abx, bcx$ or $acx \in I$.
On the other hand, if $cx^{2} \in I$, then $acx(x+ b) \in I$ implies 
that one of $acx, ac(x+b), cx(x+b),$ or $ax(x+b) \in I$. In any of these cases, we can deduce either $abx, bcx$ or $acx \in I$,
and we're done. \end{proof}

\begin{coro}
Suppose that $I$ is a $3$-absorbing ideal of a ring $R$ and $\sqrt{I} = P$ is prime. If $x,y,z \in
P$, then either $I_{xz} \subseteq I_{xy}$ or $I_{xy} \subseteq I_{xz}$.
Furthermore, $I_{xy}$ is $1$-absorbing.
\end{coro}

\begin{proof}
We can assume $xy,xz \notin I$, otherwise there's nothing to do since $I_{xy} = I_{xz} = R$. We have
that $I_{x}$ is $2$-absorbing by the previous result, so that the set
$I_{xa} = \{ (I :_{R} xa) \mid a \in \sqrt{I_{x}} \backslash I_{x} \}$
is a totally ordered set of $1$-absorbing ideals \cite[Theorems
2.5,2.6]{Badawi}. Since $z,y \in \sqrt{I} \subseteq \sqrt{I_{x}}$ and $z,y \notin I_{x}$ by our assumption, 
the claim follows.
\end{proof}

\section*{Acknowledgements}

The authors thank Youngsu Kim and Paolo Mantero for their helpful comments and suggestions writing this note.


\begin{thebibliography}{99}

\bibitem{Anderson}
D. F. Anderson and A. Badawi,
{\em On $n$-absorbing ideals of commutative rings},
Comm. in Alg., {\bf 39(5)}, 1646-1672 (2011).

\bibitem{Badawi}
Ayman Badawi,
{\em On $2$-absorbing ideals of commutative rings},
Bull. Austral. Math. Soc., {\bf 75(3)}, 417-429 (2007).

\end{thebibliography}
\end{document}